\documentclass[amsfonts,12pt]{amsart}

\usepackage{epsfig}

\newtheorem{thm}{Theorem}[section]

\newtheorem{lem}[thm]{Lemma}
\newtheorem{cor}[thm]{Corollary}

\newtheorem{prp}[thm]{Proposition}
\newtheorem{ex}[thm]{Example}

\newcommand{\st}{{\mathrm{st}}\,}

\newcommand{\Volcic}{Vol\v{c}i\v{c}}
\newcommand{\ee}{\varepsilon}

\newcommand{\R}{\mathbb{R}}

\newcommand{\N}{\mathbb{N}}

\begin{document}
\hfill\today
\bigskip

\title{CHARACTERIZING THE DUAL MIXED VOLUME VIA ADDITIVE FUNCTIONALS}
\author[Paolo Dulio, Richard J. Gardner, and Carla Peri]
{Paolo Dulio, Richard J. Gardner, and Carla Peri}
\address{Dipartimento di Matematica ``F.~Brioschi",
Politecnico di Milano, Piazza Leonardo da Vinci 32, I-20133
Milano, Italy} \email{paolo.dulio@polimi.it}
\address{Department of Mathematics, Western
Washington University, Bellingham, WA 98225-9063}
\email{Richard.Gardner@wwu.edu}
\address{Universit\'{a} Cattolica del Sacro Cuore\\
Via Emilia Parmense 84\\
29122 Piacenza, Italy}
\email{carla.peri@unicatt.it}
\thanks{First author supported in
part by U.S.~National Science Foundation grant DMS-1103612}
 \maketitle
\begin{abstract}
Integral representations are obtained of positive additive functionals on finite products of the space of continuous functions (or of bounded Borel functions) on a compact Hausdorff space.  These are shown to yield characterizations of the dual mixed volume, the fundamental concept in the dual Brunn-Minkowski theory. The characterizations are shown to be best possible in the sense that none of the assumptions can be omitted.  The results obtained are in the spirit of a similar characterization of the mixed volume in the classical Brunn-Minkowski theory, obtained recently by Milman and Schneider, but the methods employed are completely different.
\end{abstract}

\section{Introduction}

At the core of modern convex geometry lies the Brunn-Minkowski theory.  This powerful apparatus, constructed by Minkowski, Blaschke, Aleksandrov, Fenchel, and many others, provides a framework within which questions concerning the metrical properties of convex bodies in Euclidean space $\R^n$ may be formulated and attacked.  The theory arises from combining two notions, volume and vector or Minkowski addition, defined between sets $A$ and $B$ by $A+B=\{x+y: x\in A,\,y\in B\}$.  Both are ingredients in Minkowski's theorem on mixed volumes, which states that if $K_1,\dots,K_m$ are compact convex sets in $\R^n$, and $t_1,\dots,t_m\ge 0$, the volume ${{\mathcal{H}}^n}(t_1K_1+\cdots +t_mK_m)$ is a homogeneous polynomial of degree $n$ in the variables $t_1,\dots,t_m$. (See Section~\ref{prelim} for unexplained notation and terminology.)  The coefficients in this polynomial are called mixed volumes. When $m=n$, $K_1=\cdots=K_i=K$, and $K_{i+1}=\cdots=K_n=B^n$, the unit ball in $\R^n$, then, up to constant factors, the mixed volumes turn out to be averages of volumes of orthogonal projections of $K$ onto subspaces, and include the volume, surface area, and mean width of $K$, as special cases. The classic treatise of Schneider \cite{Sch93} provides a detailed survey of the Brunn-Minkowski theory and many references there give testament to the wide variety of its applications in science.

In the last few decades, the Brunn-Minkowski theory has been extended in several important ways.  One such extension, now called the dual Brunn-Minkowski theory, arose from the 1975 observation of Lutwak \cite{L1} that if $K_1,\dots,K_m$ are star sets (bounded Borel sets star-shaped at the origin $o$) in $\R^n$, and $t_1,\dots,t_m\ge 0$, the volume ${{\mathcal{H}}^n}(t_1K_1\widetilde{+}\cdots\widetilde{+}t_mK_m)$ is a homogeneous polynomial of degree $n$ in the variables $t_1,\dots,t_m$.  (Here $\widetilde{+}$ denotes radial addition; one defines $x\widetilde{+}y=x+y$ if $x$, $y$, and $o$ are collinear, $x\widetilde{+}y=o$, otherwise, and
$$L\widetilde{+}M=\{x\widetilde{+}y: x\in L,\, y\in M\},$$
for star sets $L$ and $M$ in $\R^n$.) Lutwak called the coefficients of this polynomial dual mixed volumes and showed that up to constant factors, they include averages of volumes of intersections of a star set with (linear) subspaces.

There is a perfect analogy between Minkowski's theorem for mixed volumes and Lutwak's theorem for dual mixed volumes.  Such analogies, not always quite so perfect, between results and concepts in the Brunn-Minkowski theory and its dual, have often been observed, but this duality has not yet been fully explained.

Since 1975, the dual Brunn-Minkowski theory has seen a dramatic development, for example providing the tools for the solution of the Busemann-Petty problem in \cite{G3}, \cite{GKS}, \cite{L5}, and \cite{Z1}. It also has connections and applications to integral geometry, Minkowski geometry, the local theory of Banach spaces, geometric tomography, and stereology; see \cite{Gar06} and the references given there.

An extremely productive recent trend in convex geometry involves the characterization of useful concepts via a few of their properties.  For example, the operations of Minkowski and radial addition mentioned above were characterized in \cite{GHW1}.  These characterizations indicate the fundamental nature of these two operations in geometry and led to fresh insights into the nature of the Brunn-Minkowski theory and its possible extensions in \cite{GHW2}.  Earlier, Milman and Schneider \cite{MilS11} gave several results characterizing the mixed volume $V(K_1,\dots,K_n)$ of compact convex sets $K_1,\dots,K_n$ in $\R^n$, $n\ge 2$.  For example, \cite[Theorem~2]{MilS11} states (in a slightly more general form):

\smallskip

{\bf Theorem A} (Milman and Schneider).  {\em Let $F:\left({\mathcal{K}}_s^n\right)^n\to\R$, $n\ge 2$, be an additive, increasing functional on the class of $n$-tuples of centrally symmetric compact convex sets in $\R^n$.  If $F$ vanishes whenever two of its arguments are parallel line segments, then there is a $c\ge 0$ such that
$$F(K_1,\dots,K_n)=c\,V(K_1,\dots,K_n),$$
for all $K_1,\dots,K_n\in {\mathcal{K}}_s^n$.}

\smallskip

Here ``additive" means additive in each argument with respect to Minkowski addition, and ``increasing" means increasing with respect to set inclusion in each argument.  Thus this striking result characterizes the mixed volume via just three very simple properties, none of which may be omitted.

The principal goal of this paper is to establish a corresponding characterization of the dual mixed volume $\widetilde{V}(L_1,\dots,L_n)$ of star sets $L_1,\dots,L_n$ in $\R^n$, $n\ge 2$.  One of our main results (see Theorem~\ref{thm2}(iii) below) is as follows.

\smallskip

{\bf Theorem B}. {\em Let $F:\left({\mathcal{S}}^n\right)^n\to [0,\infty)$, $n\ge 2$, be an additive functional on the class of $n$-tuples of star sets in $\R^n$.  If $F$ is rotation invariant and vanishes whenever the intersection of two of its arguments is $\{o\}$, then there is a $c\ge 0$ such that
$$F(L_1,\dots,L_n)=c\,\widetilde{V}(L_1,\dots,L_n),$$
for all $L_1,\dots,L_n\in {\mathcal{S}}^n$.}

\smallskip

Here ``additive" means additive in each argument with respect to radial addition and ``rotation invariant" means that $F$ is unchanged if the same rotation of $S^{n-1}$ is applied to each of its arguments (see (\ref{staraddd}) and (\ref{starrotd}) below). In Examples~\ref{example1}, \ref{example2}, and~\ref{example3}, we show that none of the properties of $F$ assumed in Theorem~B can be omitted.

There is a strong similarity between Theorems A and B, another instance of the still unexplained duality mentioned above.  It is perhaps more instructive to comment on the differences between the two results.  Firstly, Theorem~B does not require symmetry of the sets concerned, as Theorem~A does. In fact, the role of symmetry in Theorem~A has not been completely resolved; see \cite[p.~672]{MilS11}.  Secondly, the functional $F$ in Theorem~B is assumed positive, while that in Theorem~A is real valued.  In Theorem~\ref{thm2real} below we actually provide a version of Theorem~B for real-valued functionals, and then have to assume that $F$ is also increasing, as in Theorem~A.  This highlights a third and important difference, namely, the quite strong assumption of rotation invariance in Theorem~B.  However, this seems unavoidable (see Example~\ref{example3}).  Moreover, we prove other results that pinpoint the role of rotation invariance; in particular, Theorem~\ref{thm2}(ii) completely characterizes all functionals $F$ satisfying the other hypotheses of Theorem~B but not rotation invariance.

If, as well as rotation invariance, the assumption on $F$ in Theorem~B that it vanishes when the intersection of two of its arguments is $\{o\}$ is also omitted, then Theorem~\ref{thm2}(i) states that there is a finite Radon measure $\mu$ in $\left(S^{n-1}\right)^n$ such that
\begin{equation}\label{Faddintro}
F(L_1,\dots,L_n)=\int_{\left(S^{n-1}\right)^n}\rho_{L_1}(u_1)\cdots \rho_{L_n}(u_n)\,d\mu(u_1,\dots,u_n),
\end{equation}
for all $L_1,\dots, L_n\in {\mathcal{S}}^n$.  Here $\rho_{L_i}$ denotes the radial function of $L_i$, the function giving for all $u\in S^{n-1}$ the distance from the origin to the boundary of $L_i$ in the direction $u$.  It is remarkable that only the additivity and positivity of $F$ are required for (\ref{Faddintro}) to hold.

Radial functions of star sets are just nonnegative bounded Borel functions on the unit sphere, so it is natural to view (\ref{Faddintro}) in the context of positive additive functionals on finite products of $B_+(X)$, the class of nonnegative bounded Borel functions on a compact Hausdorff space $X$. This is in fact the approach we take, and the corresponding result, more general than (\ref{Faddintro}), is stated in Theorem~\ref{thm3}.  The latter is in turn derived from a similar result, Theorem~\ref{thm1}, in which $B_+(X)$ is replaced by $C_+(X)$, the class of nonnegative continuous functions on a compact Hausdorff space $X$.  Just as Theorem~\ref{thm3} yields statements about positive additive functionals on $n$-tuples of star sets in $\R^n$, so Theorem~\ref{thm1} yields statements about positive additive functionals on $n$-tuples of star bodies in $\R^n$.  Star bodies are star sets with continuous radial functions and they have found many uses in the dual Brunn-Minkowski theory.  Theorem~B and related results hold when star sets are replaced by star bodies; see Theorem~\ref{thm2}.

Our results about positive additive functionals on finite products of $B_+(X)$ or $C_+(X)$, where $X$ is a compact Hausdorff space, may have some independent interest.  Of course, there is already much information in this direction in the literature.  In fact, we use a result of Stach\'{o} \cite[Theorem~7.1]{Sta08}, stated in Proposition~\ref{prp0}, as a springboard for our work.  Stach\'{o}'s theorem provides an integral representation for a continuous positive multilinear functional on $C(X)^n$, where $X$ is a locally compact Hausdorff space.  We are grateful to Laszlo Stach\'{o} for supplying Proposition~\ref{prp000} and to him and Fernando Bombal for helpful correspondence concerning results such as \cite[Theorem~7.1]{Sta08}, which arise in the representation of polymeasures.

The paper is organized as follows. After the preliminary Section~\ref{prelim}, we focus in Section~\ref{funsonPV} on additive functionals on finite products of partially ordered vector spaces.  The results are applied in Section~\ref{funsonCB}, where our results on positive additive functionals on $n$-tuples of continuous functions or of bounded Borel functions are proved.  These are applied in turn in Section~\ref{SorS}, which contains our characterizations of the dual mixed volume in terms of positive additive functionals.  The short Section~\ref{Direct} deals with real-valued additive functionals and in the Appendix we sketch a direct proof of a version of one of the characterizations of the dual mixed volume.

\section{Definitions and preliminaries}\label{prelim}

As usual, $S^{n-1}$ denotes the unit sphere and $o$ the origin in Euclidean
$n$-space $\R^n$, where shall assume that $n\ge 2$ throughout.

We denote by $1_X$ the characteristic function of any set $X$.  If $X$ is a subset of $\R^n$ and $t\in\R$, then $tX=\{tx:x\in X\}$.

As usual, $C(X)$ denotes the space of continuous real-valued functions on a topological space $X$ equipped with the $L_{\infty}$ norm.  We denote the set of bounded Borel functions on $X$ by $B(X)$.  Then $C_+(X)$ and $B_+(X)$ are the sets of nonnegative functions in $C(X)$ and $B(X)$, respectively.

We write ${\mathcal{H}}^k$ for $k$-dimensional Hausdorff measure in $\R^n$,
where $k\in\{1,\dots, n\}$. The notation $dz$ will always
mean $d{\mathcal{H}}^k(z)$ for the appropriate $k=1,\dots, n$.  In particular, for integrals over $S^{n-1}$, the symbol $du$ indicates integration with respect to spherical Lebesgue measure.

We denote by $l_x$ the line through the origin containing $x\in\R^n\setminus\{o\}$. A set $L$ in $\R^n$ is {\it star-shaped at $o$} if $L\cap l_u$ is either empty
or a (possibly degenerate) closed line segment for each $u\in
S^{n-1}$. If $L$ is star-shaped at $o$, we define its {\it radial
function} $\rho_L$ for $x\in \R^n\setminus\{o\}$ by
$$\rho_{L}(x)=\left\{\begin{array}{ll} \max\{c:cx\in L\}, &
\mbox{if $L\cap l_x\neq\emptyset$,}\\ 0, & \mbox{otherwise.}
\end{array}\right.
$$
This definition is a slight modification of \cite[(0.28)]{Gar06};
as defined here, the domain of $\rho_L$ is always $\R^n\setminus\{o\}$.  Radial functions are {\em homogeneous of degree $-$1}, that is,
$$\rho_L(rx)=r^{-1}\rho_L(x),$$
for all $x\in \R^n\setminus\{o\}$ and $r>0$, and this allows us to regard them as functions on $S^{n-1}$.

In this paper, a {\it star set} in $\R^n$ is a bounded Borel set that is star-shaped at $o$ and contains $o$, and a {\em star body} in $\R^n$ is a star set with a continuous radial function.  (Other definitions have been used; see, for example, \cite[Section~0.7]{Gar06} and \cite{GV}.)  Then $L$ is a star set (or star body) in $\R^n$ if and only if $\rho_L\in B_+(S^{n-1})$ (or $\rho_L\in C_+(S^{n-1})$, respectively).  We denote the class of star sets in $\R^n$ by ${\mathcal S}^n$ and the class of star bodies in $\R^n$ by ${\mathcal S}_o^n$.
Note that ${\mathcal S}^n$ is closed under finite unions, countable intersections, and intersections with (linear) subspaces.

The \emph{radial sum} $L=L_1\widetilde{+}\cdots\widetilde{+}L_m$ of $L_i\in {\mathcal S}^n$, $i=1,\dots,m$, is the star set with radial function
$$\rho_L=\rho_{L_1}+\cdots+\rho_{L_m}.$$

We recall the basics of Lutwak's dual Brunn-Minkowski theory.  Lutwak \cite{L1} worked with star bodies containing $o$ in their interiors, but it was noted in \cite{GVV} that with appropriate minor modifications, many of his definitions and results extend immediately to the class ${\mathcal S}^n$.  In particular, we can define the {\em dual mixed volume } $\widetilde{V}(L_1,\ldots,L_n)$ of
sets $L_1,\dots,L_n\in{\mathcal S}^n$ by
\begin{equation}\label{dmv}
\widetilde{V}(L_1,\ldots,L_n)=\frac1n\int_{S^{n-1}}\rho_{L_1}(u)
\rho_{L_2}(u)\cdots \rho_{L_n}(u)\,du.
\end{equation}

Lutwak \cite{L1} (see also \cite[Theorem~A.7.1]{Gar06}) found the
following analogue of Minkowski's theorem on mixed volumes.

\begin{prp}\label{7}
Let $L_i\in{\mathcal S}^n$, $i=1,\dots,m$.  If
$$L=t_1L_1\widetilde{+}\cdots\widetilde{+}t_mL_m,$$ where $t_i\ge
0$, then ${\mathcal{H}}^n(L)$ is a homogeneous polynomial of degree $n$ in the variables $t_i$, whose coefficients are dual mixed volumes. Specifically,
$${\mathcal{H}}^n(L)=\sum_{i_1=1}^m\cdots\sum_{i_n=1}^m\widetilde{V}(L_{i_1},\dots
,L_{i_n}) t_{i_1}\cdots t_{i_n}.$$
\end{prp}

Lutwak showed that his definition (\ref{dmv}) of the dual mixed
volume $\widetilde{V}(L_{1},\dots ,L_{n})$ is compatible with the
previous theorem, and in particular $\widetilde{V}(L,\dots,L)={\mathcal{H}}^n(L)$.

The term {\em vector space} in this paper always means a real vector space.

Let $Y$ be a vector space.  A {\em cone} in $Y$ will always mean a pointed cone, that is, a subset of $Y$ closed under multiplication by nonnegative scalars.  A {\em convex cone} is a cone that is also closed under addition.  A {\em double cone} will mean a subset closed under arbitrary scalar multiplication.  The term {\em subspace} always means a linear subspace.

Let $Y$ be a vector space and let $n\in \N$.  If $T$ is a subspace of $Y$, a functional $F:T^n\rightarrow\R$ is called {\em multilinear} if it is linear in each variable.  Let $A$ be a subset of $Y$ that is closed under addition.  We say that $F$ is {\em additive} on $A^n$ if it is additive in each variable, i.e.,
$$F(v_1,\dots,v_{i-1},v_i+w_i,v_{i+1},\dots,v_n)=
F(v_1,\dots,v_n)+F(v_1,\dots,v_{i-1},w_i,v_{i+1},\dots,v_n),$$
whenever $v_i,w_i\in A$, $i=1,\dots,n$.  Let $C$ be double cone (or a cone) in $Y$.  Then $F$ is called {\em homogeneous} (or {\em positively homogeneous}, respectively) (of degree $1$) on $C^n$ if
$$F(t_1v_1,\dots, t_nv_n)=t_1\cdots t_n F(v_1,\dots,v_n),$$
for all $t_i\in \R$ (or $t_i\ge 0$, respectively) and $v_i\in C$, $i=1,\dots,n$.  Clearly, a functional is multilinear if and only if it is both additive and homogeneous.

Let $T$ be a subspace of $Y$. If $F:T^n\rightarrow\R$ is additive, then $F$ vanishes whenever one of its arguments is the zero vector.  Using this, it is easy to see that
\begin{equation}\label{newminus}
F(v_1,\dots,v_{i-1},-v_i,v_{i+1},\dots,v_n)=-F(v_1,\dots,v_{i-1},v_i,v_{i+1},\dots,v_n),
\end{equation}
whenever $v_i\in T$, $i=1,\dots,n$.  As a consequence, any positively homogeneous additive functional on $T^n$ is homogeneous. With (\ref{newminus}) in hand, it is straightforward to show that any additive functional $F$ on $T^n$ satisfies
\begin{equation}\label{Graf}
F(v_1,\dots,v_n)-F(w_1,\dots,w_n)=\sum_{i=1}^nF(w_1,\dots,w_{i-1},v_i-w_i,v_{i+1},
\dots,v_n),
\end{equation}
whenever $v_i, w_i\in T$, $i=1,\dots,n$, where the summands on the right-hand side are $F(v_1-w_1,v_2,\dots,v_n)$ when $i=1$ and $F(w_1,\dots,w_{n-1},v_n-w_n)$ when $i=n$.

Let $Y$ be a partially ordered vector space. The {\em positive cone} $Y_+$ of $Y$ is
$$Y_+=\{v\in Y:v\ge 0\}.$$
Let $n\in \N$ and let $E$ be a subset of $Y$. A functional is called {\em positive} on $E^n$ if $F(v_1,\dots,v_n)\ge 0$ whenever $v_i\in E\cap Y_+$, $i=1,\dots,n$, and {\em increasing} on $E^n$ if whenever $v_i,w_i\in E$ and $v_i\ge w_i$, $i=1,\dots,n$, we have
$$F(v_1,\dots,v_n)\ge F(w_1,\dots,w_n).$$

We call a functional $F$ defined on $E^n$, where $E$ is a set of functions on $S^{n-1}$, {\em rotation invariant} if
$$F(\phi f_1,\dots,\phi f_n)=F(f_1,\dots,f_n),$$
for all $f_1,\dots,f_n\in E$ and rotations $\phi$ of $S^{n-1}$.  Here $(\phi f_i)(u)=f_i(\phi^{-1}u)$, for all $u\in S^{n-1}$.

\begin{prp}\label{prp0}
{\em (Stach\'{o} \cite[Theorem~7.1]{Sta08})} Let $X$ be a locally compact Hausdorff space, let $n\in \N$, and let $F$ be a continuous positive multilinear functional on $C(X)^n$.  Then there is a finite Radon measure $\mu$ in $X^n$ such that
$$
F(f_1,\dots,f_n)=\int_{X^n}f_1(x_1)\cdots f_n(x_n)\,d\mu(x_1,\dots,x_n),
$$
for all $(f_1,\dots,f_n)\in C(X)^n$.
\end{prp}

The proof of \cite[Theorem~7.1]{Sta08}, which appears in an appendix to that paper, is completely independent of the rest of \cite{Sta08} and accessible, though there are a few misprints: $\phi$ should be $\Phi$ in the first line of the statement of the theorem and on line 5 of page 21, and the equality sign on line 4 of page 21 should be less than or equal to.  A key tool is the Alaoglu-Bourbaki theorem. In the context of interest here, when $X=S^{n-1}$, the proof is somewhat simpler, since $X$ is a compact metric space. In particular, the standard Banach-Alaoglu theorem is sufficient.  It should be noted that Stach\'{o} writes at the end of the introduction of \cite{Sta08} that \cite[Theorem~7.1]{Sta08} is contained implicitly in a result of Villanueva \cite{Vil04}, while Fernando Bombal, in private communication, points to the earlier paper \cite{BomV01}.

\section{Additive functionals on finite products of partially ordered vector spaces}\label{funsonPV}

For lack of an explicit reference, we provide the proof of the following result. It follows that of the well-known case when $n=1$, i.e., the equivalence of continuity and boundedness for linear functionals.

\begin{prp}\label{prp00}
Let $Y$ be a partially ordered normed space. A multilinear functional $F$ on $Y^n$ (or $Y_+^n$) is continuous if and only if it is bounded, i.e., there is an $M$ such that
\begin{equation}\label{boundeq}
|F(v_1,\dots,v_n)|\le M\prod_{i=1}^n\|v_i\|,
\end{equation}
for all $(v_1,\dots,v_n)\in Y^n$ (or for all $(v_1,\dots,v_n)\in Y_+^n$, respectively).
\end{prp}

\begin{proof}
Suppose that $F$ is bounded. If $v_i,w_i\in Y$, $i=1,\dots,n$, then the multilinearity of $F$ yields
$$F(v_1+w_1,\dots,v_n+w_n)=\sum_{\substack{r_i\in\{0,1\}\\i\in\{1,\dots,n\}}}
F(r_1v_1+(1-r_1)w_1,\dots,r_nv_n+(1-r_n)w_n).$$
Using (\ref{boundeq}), we obtain
\begin{eqnarray*}\lefteqn{|F(v_1+w_1,\dots,v_n+w_n)-F(v_1,\dots,v_n)|}\\
&=&
\left|\sum_{\substack{r_i\in\{0,1\}\\(r_1,\dots,r_n)\neq (1,\dots,1)\\i\in\{1,\dots,n\}}}F(r_1v_1+(1-r_1)w_1,\dots,r_nv_n+(1-r_n)w_n)
\right|\\
&\leq&\sum_{\substack{r_i\in\{0,1\}\\(r_1,\dots,r_n)\neq (1,\dots,1)\\i\in\{1,\dots,n\}}} \left|F\left(r_1v_1+(1-r_1)w_1,\dots,r_nv_n+(1-r_n)w_n\right)\right|\\
&\leq& M\max\{\|w_i\|: i=1,\dots,n\}\sum_{j=1}^n\prod_{\substack{r_i\in\{0,1\}\\i\in\{1,\dots,n\}
\\i\neq j}}\|r_iv_i+(1-r_i)w_i\|.
\end{eqnarray*}
Consequently,
$$\lim_{(w_1,\dots,w_n)\to (0,\dots,0)} |F(v_1+w_1,...,v_n+w_n)-F(v_1,...,v_n)|=0,$$
proving that $F$ is continuous.

Conversely, let $F$ be continuous. Then $F$ is continuous at $(0,\dots,0)$, so for all $\ee>0$, there exists a $\delta>0$ such that
$|F(w_1,\dots,w_n)|=|F(w_1,\dots,w_n)-F(0,\dots,0)|<\ee$, for all $(w_1,\dots,w_n)$ with $\|w_i\|\leq\delta$, $i=1,\dots,n$.  Using the homogeneity of $F$, it follows that whenever $v_i\neq 0$, $i=1,\dots,n$, we have
\begin{eqnarray*}|F(v_1,\dots,v_n)|
&=&\left|F\left(\delta\|v_1\|\frac{v_1}{\delta
\|v_1\|},\dots,\delta\|v_n\|\frac{v_n}
{\delta\|v_n\|}\right)\right|\\
&=&\frac{1}{\delta^n}\prod_{i=1}^n\|v_i\|\left|F\left(\delta\frac{v_1}
{\|v_1\|},\dots,\delta\frac{v_n}{\|v_n\|}\right)\right|<
\frac{\ee}{\delta^n}\prod_{i=1}^n\|v_i\|,
\end{eqnarray*}
for all $v_1,\dots,v_n\in Y$.  If $v_i=0$ for some $i=1,\dots,n$, the previous inequality holds trivially. This shows that $F$ is bounded.

The same proof applies when $Y$ is replaced by $Y_+$.
\end{proof}

\begin{lem}\label{addmon}
Let $Y$ be a partially ordered vector space and let $n\in \N$.  If $F:Y_+^n\to [0,\infty)$ is additive, then it is increasing on $Y_+^n$.
\end{lem}

\begin{proof}
Let $v_i, w_i\in Y_+$ be such that $v_i\ge w_i$, $i=1,\dots,n$.  Define $u_i\in Y_+$ by setting $u_i=v_i-w_i$, for $i=1,\dots,n$.  Then $v_i=u_i+w_i$, $i=1,\dots,n$, so using the additivity and the fact that $F\ge 0$ on $Y_+^n$, we obtain
$$
F(v_1,\dots,v_n)=F(u_1+w_1,\dots,u_n+w_n)\ge F(w_1,\dots,w_n).
$$
\end{proof}

\begin{lem}\label{addabs22}
Let $Y$ be a Riesz space and let $n\in \N$.  If $F$ is a positive multilinear functional on $Y^n$, then
\begin{equation}\label{newabs22}
|F(v_1,\dots,v_n)|\le F(|v_1|,\dots,|v_n|),
\end{equation}
whenever $v_i\in Y$, $i=1,\dots,n$.
\end{lem}

\begin{proof}
Since $Y$ is a Riesz space, i.e., a partially ordered vector space where the order structure is a lattice, for each $i\in \{1,\dots,n\}$ we have $v_i=v_i^+-v_i^-$, where $v_i^+$ and $v_i^-$ are the positive and negative parts of $v_i$. Then $v_i^+,v_i^-\ge 0$ and $|v_i|=v_i^++v_i^-$. Using the multilinearity of $F$, we obtain
\begin{eqnarray*}
F(|v_1|,\dots,|v_n|)&=&F(v_1^++v_1^-,\dots,v_n^++v_n^-)\\
&=&
\sum_{\substack{r_i\in\{0,1\}\\i\in\{1,\dots,n\}}}
F(r_1v_1^++(1-r_1)v_1^-,\dots,r_nv_n^++(1-r_n)v_n^-)
\end{eqnarray*}
and
\begin{eqnarray*}
F(v_1,\dots,v_n)&=&F(v_1^+-v_1^-,\dots,v_n^+-v_n^-)\\
&=&
\sum_{\substack{r_i\in\{0,1\}\\i\in\{1,\dots,n\}}}
(-1)^{(1-r_1)+\dots+(1-r_n)}F(r_1v_1^++(1-r_1)v_1^-,\dots,r_nv_n^++(1-r_n)v_n^-).
\end{eqnarray*}
Since all the arguments of $F$ in the previous two sums belong to $Y^+$, the positivity of $F$ implies that $F(v_1,\dots,v_n)\le F(|v_1|,\dots,|v_n|)$.  Similarly, $-F(v_1,\dots,v_n)=F(-v_1,v_2,\dots,v_n)\le F(|v_1|,\dots,|v_n|)$.
\end{proof}

The special case $Y=C(X)$ of the following result was communicated to us by Laszlo Stach\'{o}, who stated that the argument is ``rather standard".

\begin{prp}\label{prp000}
Let $Y$ be a normed Riesz space and let $F$ be a positive multilinear functional on $Y^n$ such that $F$ is continuous on $Y_+^n$. Then $F$ is continuous.
\end{prp}

\begin{proof}
Suppose that $F$ is not continuous on $Y^n$. By Proposition~\ref{prp00}, for each $k\in\N$, there are $ v^{(k)}_1,\ldots, v^{(k)}_n\in Y$ such that
$$
\left|F\left(v^{(k)}_1 ,\ldots,v^{(k)}_n\right)\right|\ge 2^{kn}
\prod_{i=1}^n \|v_i^{(k)}\|.
$$
Replacing $v^{(k)}_i$ by $v^{(k)}_i/\|v_i^{(k)}\|$ and using the fact that $F$ is positively homogeneous of degree 1, we may assume that this holds with $\|v_i^{(k)}\|=1$ for each $i=1,\dots,n$.  Then, using (\ref{newabs22}), we have
\begin{equation}\label{obs}
2^{kn} \le \left|F\left(v^{(k)}_1 ,\ldots,v^{(k)}_n\right)\right| \le
F\left(| v^{(k)}_1|,\ldots,| v^{(k)}_n|\right).
\end{equation}
For $i=1,\dots,n$, define
$$w_i= \sum_{k=1}^\infty 2^{-k}|v^{(k)}_i|.
$$
Then $w_i \in Y_+$ and $\|w_i\|\le 1$. (For the latter inequality, note that by the definition of a Riesz norm, $|u|\le |v|$ implies $\|u\|\le \|v\|$ for all $u,v\in Y$.  Applying this with $u=|v^{(k)}_i|$ and $v=v^{(k)}_i$, we obtain $\| |v^{(k)}_i|\|\le \|v^{(k)}_i\|$ for each $i=1,\dots,n$ and $k\in \N$.)  Since $F$ is positive and multilinear, it is also increasing on $Y_+^n$, by Lemma~\ref{addmon}.  Using these facts and (\ref{obs}), we obtain, for all $m\in \N$,
\begin{eqnarray*}
F\left(w_1,\ldots, w_n\right) & \ge &
F\left(\sum_{k=1}^m 2^{-k} |v^{(k)}_1 |, \ldots,
\sum_{k=1}^m 2^{-k} |v^{(k)}_n|\right)\\
& \ge &
\sum_{k=1}^m 2^{-kn} F\left(|v^{(k)}_1|,\ldots,|v^{(k)}_n|\right)
\ge  m\ge m \prod_{i=1}^n \|w_i\|.
\end{eqnarray*}
By Proposition~\ref{prp00}, $F$ is not continuous on $Y_+^n$, a contradiction.
\end{proof}

The argument in the next lemma is standard and in our context goes back at least to Firey \cite{Fir76}; this paper is referred to in the proof of \cite[Lemma~1]{MilS11}. In fact, Firey himself seems to refer to Aleksandrov \cite{Ale37a}.

\begin{lem}\label{addhomon}
Let $Y$ be a partially ordered vector space and let $n\in \N$.  If $F:Y_+^n\to [0,\infty)$ is additive, then it is positively homogeneous.
\end{lem}

\begin{proof}
Let $v_i\in Y_+$, $i=1,\dots,n$.  To prove that $F$ is positively homogeneous, it will suffice to show that if $t\ge 0$, then
$$F(tv_1,v_2\dots,v_n)= tF(v_1,\dots,v_n).$$
To see this, let $p\in \N$.  Since $pv_1=v_1+\cdots +v_1$, where the sum involves $p$ copies of $v_1$, the additivity of $F$ implies that
$$F(pv_1,v_2\dots,v_n)=F(v_1+\cdots +v_1,v_2\dots,v_n)= pF(v_1,\dots,v_n).$$
Therefore if $p,q\in \N$, then
$$qF((p/q)v_1,v_2\dots,v_n)=
pqF((1/q)v_1,v_2\dots,v_n)=
pF(v_1,\dots,v_n),$$
which yields
$$F((p/q)v_1,v_2,\dots,v_n)= (p/q)F(v_1,\dots,v_n).$$
Thus $F$ is positively homogeneous for rational factors.  Now let $t\ge 0$.  Let $(r_m)$ and $(s_m)$, $m\in \N$, be increasing (and decreasing, respectively) sequences of nonnegative rational numbers such that $r_m\rightarrow t$ and $s_m\rightarrow t$ as $m\rightarrow\infty$.  Using the positive homogeneity for rational factors and the fact that $F$ is increasing (a consequence of Lemma~\ref{addmon}), we obtain, for $m\in \N$,
\begin{eqnarray*}
r_mF(v_1,\dots,v_n)&=&F(r_mv_1,v_2,\dots,v_n)\\
&\le & F(tv_1,v_2,\dots,v_n)\\
&\le & F(s_mv_1,v_2,\dots,v_n)
=s_mF(v_1,\dots,v_n).
\end{eqnarray*}
Letting $m\rightarrow\infty$, we obtain
$$F(tv_1,v_2,\dots,v_n)=tF(v_1,\dots,v_n).$$
This completes the proof.
\end{proof}

\begin{lem}\label{lemthm1}
Let $Y$ be a Riesz space and let $n\in \N$. If $F:Y_+^n\to [0,\infty)$ is additive, there is an extension of $F$ to a positive multilinear functional on $Y^n$.
\end{lem}

\begin{proof}
Suppose that $v_i\in Y$, $i=1,\dots,n$.  Since $Y$ is a Riesz space, we can write $v_i=v_i^+-v_i^-$, for each $i$, where $v_i^+\in Y_+$ and $v_i^-\in Y_+$ are the positive and negative parts of $v_i$.  Now if $v_i=v_i^{(0)}-v_i^{(1)}$, where $v_i^{(0)}, v_i^{(1)}\in Y_+$, for $i=1,\dots,n$,
we define
\begin{equation}\label{defi}
F(v_1,\dots,v_n)=\sum_{r_1,\dots,r_n\in\{0,1\}} (-1)^{r_1+\dots+r_n}F\left(v_1^{(r_1)},\dots,v_n^{(r_n)}\right).
\end{equation}

We claim that the extension of $F$ to $Y^n$ via (\ref{defi}) is well defined.  (The definition (\ref{defi}) and the following argument are also standard and analogous to those in \cite[p.~285]{Sch93}, for example.)   To see this, let $i\in\{1,\dots,n\}$ and suppose that
$$v_i=v_i^{(0)}-v_i^{(1)}=w_i^{(0)}-w_i^{(1)},$$
where $v_i^{(0)}, v_i^{(1)},w_i^{(0)}, w_i^{(1)}\in Y_+$.  Then
$$v_i^{(0)}+w_i^{(1)}=w_i^{(0)}+v_i^{(1)}.$$
Therefore
\begin{eqnarray*}
\lefteqn{F(v_1,\dots,v_{i-1},v_i^{(0)},v_{i+1},\dots,v_n)+
F(v_1,\dots,v_{i-1},w_i^{(1)},v_{i+1},\dots,v_n)}\\
&=& \sum_{r_j\in\{0,1\},\,j\neq i;\,r_i=0} (-1)^{r_1+\dots+r_n}F\left(v_1^{(r_1)},\dots,
v_{i-1}^{(r_{i-1})},v_i^{(0)},v_{i+1}^{(r_{i+1})},\dots,v_n^{(r_n)}\right)\\
&  &+\sum_{r_j\in\{0,1\},\, j\neq i;\, r_i=0} (-1)^{r_1+\dots+r_n}F\left(v_1^{(r_1)},\dots,
v_{i-1}^{(r_{i-1})},w_i^{(1)},v_{i+1}^{(r_{i+1})},\dots,v_n^{(r_n)}\right)\\
& =& \sum_{r_j\in\{0,1\},\, j\neq i; \,r_i=0} (-1)^{r_1+\dots+r_n}F\left(v_1^{(r_1)},\dots,
v_{i-1}^{(r_{i-1})},v_i^{(0)}+w_i^{(1)},v_{i+1}^{(r_{i+1})},\dots,
v_n^{(r_n)}\right)\\
& =& \sum_{r_j\in\{0,1\},\, j\neq i; \,r_i=0} (-1)^{r_1+\dots+r_n}F\left(v_1^{(r_1)},\dots,
v_{i-1}^{(r_{i-1})},w_i^{(0)}+v_i^{(1)},v_{i+1}^{(r_{i+1})},
\dots,v_n^{(r_n)}\right)\\
&=& F(v_1,\dots,v_{i-1},w_i^{(0)},v_{i+1},\dots,v_n)+
F(v_1,\dots,v_{i-1},v_i^{(1)},v_{i+1},\dots,v_n),\\
\end{eqnarray*}
which yields
\begin{eqnarray*}
\lefteqn{F(v_1,\dots,v_{i-1},v_i^{(0)},v_{i+1},\dots,v_n)-
F(v_1,\dots,v_{i-1},v_i^{(1)},v_{i+1},\dots,v_n)}\\
& = & F(v_1,\dots,v_{i-1},w_i^{(0)},v_{i+1},\dots,v_n)-
F(v_1,\dots,v_{i-1},w_i^{(1)},v_{i+1},\dots,v_n).
\end{eqnarray*}
This suffices to prove the claim.

Note that the extension of $F$ defined by (\ref{defi}) is positive on $Y^n$, since it is positive on $Y_+^n$.

Next, we claim that $F$ as defined by (\ref{defi}) is multilinear.  Indeed, let $i\in\{1,\dots,n\}$, $v_i=v_i^{(0)}-v_i^{(1)}$, and $w_i=w_i^{(0)}-w_i^{(1)}$, where $v_i^{(0)}, v_i^{(1)}, w_i^{(0)}, w_i^{(1)}\in Y_+$.  Then
$$v_i+w_i=\left(v_i^{(0)}+w_i^{(0)}\right)-\left(v_i^{(1)}+w_i^{(1)}\right),$$ so
\begin{eqnarray*}
\lefteqn{F(v_1,\dots,v_{i-1},v_i+w_i,v_{i+1},\dots,v_n)}\\
&=&
\sum_{r_1,\dots,r_n\in\{0,1\}} (-1)^{r_1+\dots+r_n}F\left(v_1^{(r_1)},
\dots,v_{i-1}^{(r_{i-1})},v_i^{(r_i)}+w_i^{(r_i)},v_{i+1}^{(r_{i+1})},
\dots,v_{n}^{(r_{n})}\right)\\
&=&
F(v_1,\dots,v_{i-1},v_i,v_{i+1},\dots,v_n)+
F(v_1,\dots,v_{i-1},w_i,v_{i+1},\dots,v_n).
\end{eqnarray*}
Therefore $F$ is additive.

Let $\alpha\in\R$.  If $\alpha\ge 0$, we have $\alpha v_i=\alpha v_i^{(0)}-\alpha v_i^{(1)}$.  Noting that each of the variables in the summands in (\ref{defi}) are vectors in $Y_+$ and using (\ref{defi}) and Lemma~\ref{addhomon}, we obtain
$$
F(v_1,\dots,v_{i-1},\alpha v_i,v_{i+1},\dots,v_n)=
\alpha F(v_1,\dots,v_{i-1},v_i,v_{i+1},\dots,v_n).
$$
Thus $F$ is positively homogeneous on $Y^n$ and therefore, as noted after (\ref{newminus}), homogeneous.  Together with the additivity, this implies that $F$ is multilinear, so the claim is proved.
\end{proof}

\section{Positive additive functionals on $\left(C_+(S^{n-1})\right)^n$ or $\left(B_+(S^{n-1})\right)^n$}\label{funsonCB}

\begin{lem}\label{lemthm2}
Let $X$ be a compact Hausdorff space and let $n\in \N$. If $F:\left(C_+(X)\right)^n\to [0,\infty)$ is additive, then it is continuous.
\end{lem}

\begin{proof}
Let $F:\left(C_+(X)\right)^n\to [0,\infty)$ be additive.  By Lemmas~\ref{addmon} and~\ref{addhomon} with $Y=C(X)$, $F$ is also increasing and positively homogeneous.
Fix $f=(f_1,\dots,f_n)\in \left(C_+(X)\right)^n$ and  suppose that $f^{(j)}=(f_1^{(j)},\dots,f_n^{(j)})$, $j\in \N$, is a sequence in $\left(C_+(X)\right)^n$ converging to $f$.  Let $0<\ee\le 1$ be given.  Choose $j_0$ such that $\|f_i-f^{(j)}_i\|_{\infty}\le\ee$ for all $j\ge j_0$ and all $i=1,\dots,n$.  Define
$$M=F(f_1+1,\dots,f_n+1)$$
and
$$M_j=\max\left\{F(h_1,\dots,h_n): h_i=f_i^{(j)}~{\text{or}}~h_i=1, ~i=1,\dots,n\right\},$$
for $j\in \N$. From the facts that $F$ is increasing and that $f_i\ge 0$ and $f_i^{(j)}\le f_i+\ee\le f_i+1$ for all $j\ge j_0$ and $i=1,\dots,n$, we obtain $M_j\le M$, for $j\ge j_0$.
Using the additivity and positive homogeneity of $F$, we can expand the quantity $F(f^{(j)}_1+\ee,\dots,f^{(j)}_n+\ee)$ into an expression involving $2^n$ terms, each of the form $\ee^kF(h_1,\dots,h_n)$, where $h_i=f_i^{(j)}$ or $h_i=1$ and $k\ge 1$ for all but one term.  Recalling that $\ee\le 1$, this gives
$$F(f) \le F(f^{(j)}_1+\ee,\dots,f^{(j)}_n+\ee)\le F(f^{(j)})+\ee (2^{n}-1)M_j\le F(f^{(j)})+\ee (2^{n}-1)M,
$$
for all $j\ge j_0$.   Similarly, we get
$$F(f^{(j)}) \le F(f_1+\ee,\dots,f_n+\ee)\le F(f)+\ee (2^{n}-1)M,
$$
for all $j\ge j_0$. This shows that $F(f^{(j)})\rightarrow F(f)$ as $j\rightarrow\infty$ and proves the claim.
\end{proof}

\begin{thm}\label{thm1}
Let $X$ be a compact Hausdorff space, let $n\in \N$, and let $F:\left(C_+(X)\right)^n\to [0,\infty)$ be additive.  Then there is a finite Radon measure $\mu$ in $X^n$ such that
\begin{equation}\label{1a}
F(f_1,\dots,f_n)=\int_{X^n}f_1(x_1)\cdots f_n(x_n)\,d\mu(x_1,\dots,x_n),
\end{equation}
for all $(f_1,\dots,f_n)\in \left(C_+(X)\right)^n$.
\end{thm}

\begin{proof}
By Lemma~\ref{lemthm1} with $Y=C(X)$, $F$ extends to a positive multilinear functional on $\left(C(X)\right)^n$ which is continuous on $\left(C_+(X)\right)^n$ by Lemma~\ref{lemthm2}.  By Proposition~\ref{prp000} with $Y=C(X)$, this extension is continuous on $\left(C(X)\right)^n$ and then by Proposition~\ref{prp0}, it has the integral representation (\ref{1a}).
\end{proof}

\begin{cor}\label{corthm1}
Let $X$ be a compact Hausdorff space and let $n\in \N$.  Suppose that $F:\left(C_+(X)\right)^n\to [0,\infty)$ is additive and vanishes when the supports of two of its arguments are disjoint. Then there is a finite Radon measure $\mu$ in $X$ such that
\begin{equation}\label{11}
F(f_1,\dots,f_n)=\int_{X}f_1(x)\cdots f_n(x)\,d\mu(x),
\end{equation}
for all $(f_1,\dots,f_n)\in \left(C_+(X)\right)^n$.
\end{cor}

\begin{proof}
By Theorem~\ref{thm1}, there is a finite Radon measure $\nu$ in $X^n$ such that (\ref{1a}) holds with $\mu$ replaced by $\nu$. Suppose that $(x_1,\dots,x_n)\in X^n$ is such that $x_{i_1}\neq x_{i_2}$ for some $1\le i_1\neq i_2\le n$.  Choose open sets $U_i$ in $X$ with $x_i\in U_i$, $i=1,\dots,n$, such that the closures of  $U_{i_1}$ and $U_{i_2}$ are disjoint.  Define $f_i\in C_+(X)$ such that $f_i(x)=1$ for all $x\in U_i$, and the supports of $f_{i_1}$ and $f_{i_2}$ are disjoint.  Then we have
\begin{eqnarray*}
0&=&F(f_1,\dots,f_n)=\int_{X^n}f_1(x_1)\cdots f_n(x_n)\,d\nu(x_1,\dots,x_n)\\
&\ge &\int_{X^n}1_{U_1}(x_1)\cdots 1_{U_n}(x_n)\,d\nu(x_1,\dots,x_n)=\nu\left(\prod_{i=1}^n U_i\right).
\end{eqnarray*}
Thus each $(x_1,\dots,x_n)\in X^n$ not on the diagonal in $X^n$ has an open neighborhood of zero $\nu$-measure. It follows that $\nu$ is concentrated on the diagonal in $X^n$.  Let $\mu$ be the projection of $\nu$ onto $X$, defined by $\mu(E)=\nu(E\times X\times\cdots\times X)$, for all Borel sets $E$ in $X$.  Then $\mu$ is a finite Radon measure in $X$ and
$$F(f_1,\dots,f_n)=\int_{X^n}f_1(x_1)\cdots f_n(x_n)\,d\nu(x_1,\dots,x_n)=\int_{X}f_1(x)\cdots f_n(x)\,d\mu(x),
$$
for all $(f_1,\dots,f_n)\in \left(C_+(X)\right)^n$, as required.
\end{proof}

\begin{cor}\label{cor2thm1}
If $F:\left(C_+(S^{n-1})\right)^n\to [0,\infty)$ is an additive, rotation invariant functional that vanishes when the supports of two of its arguments are disjoint, then there is a $c\ge 0$ such that
$$
F(f_1,\dots,f_n)=c\int_{S^{n-1}}f_1(u)\cdots f_n(u)\,du,
$$
for all $(f_1,\dots,f_n)\in \left(C_+(S^{n-1})\right)^n$.
\end{cor}

\begin{proof}
By Corollary~\ref{corthm1}, there is a finite Radon measure $\mu$ in $S^{n-1}$ such that $F$ has the integral representation (\ref{11}) with $X=S^{n-1}$. Then, if $A$ is a Borel subset of $S^{n-1}$, the rotation invariance of $F$ yields
\begin{eqnarray*}
\mu(\phi A)&=&\int_{S^{n-1}}1_{\phi A}(u)\,d\mu(u)=\int_{S^{n-1}}1_{\phi A}(u)^n\,d\mu(u)=\int_{S^{n-1}}(\phi 1_{A})(u)^n\,d\mu(u)\\
&=&F(\phi 1_A,\dots,\phi 1_A)=F(1_A,\dots,1_A)=\mu(A).
\end{eqnarray*}
Thus $\mu$ is rotation invariant and it follows from the uniqueness of Haar measure (see, for example, \cite[p.~584]{SW}) that $\mu$ is a multiple of spherical Lebesgue measure in $S^{n-1}$.
\end{proof}

\begin{lem}\label{addabs}
Let $Y$ be a Riesz space and let $n\in \N$.  If $F$ is a positive additive functional on $Y^n$, then
\begin{equation}\label{newabs}
|F(v_1,\dots,v_n)|\le F(v_1,\dots,v_{i-1},|v_i|,v_{i+1},\dots,v_n),
\end{equation}
whenever $v_i\in Y$ and $v_j\in Y_+$, $j=1,\dots,n$, $j\neq i$.
\end{lem}

\begin{proof}
Using (\ref{newminus}), we have
\begin{eqnarray*}
\lefteqn{-F(v_1,\dots,v_n)+F(v_1,\dots,v_{i-1},|v_i|,v_{i+1},\dots,v_n)}\\
&=&
F(v_1,\dots,v_{i-1},-v_i,v_{i+1},\dots,v_n)+
F(v_1,\dots,v_{i-1},|v_i|,v_{i+1},\dots,v_n)\\
&=&
F(v_1,\dots,v_{i-1},-v_i+|v_i|,v_{i+1},\dots,v_n)\ge 0,
\end{eqnarray*}
since $-v_i+|v_i|\ge 0$.  Therefore
$F(v_1,\dots,v_n)\le F(v_1,\dots,v_{i-1},|v_i|,v_{i+1},\dots,v_n)$
and similarly one obtains
$-F(v_1,\dots,v_n)\le F(v_1,\dots,v_{i-1},|v_i|,v_{i+1},\dots,v_n)$.
\end{proof}

\begin{thm}\label{thm3}
Let $X$ be a compact Hausdorff space, let $n\in \N$, and let $F:\left(B_+(X)\right)^n\to [0,\infty)$ be additive.  Then there is a finite Radon measure $\mu$ in $X^n$ such that
\begin{equation}\label{1aaa}
F(f_1,\dots,f_n)=\int_{X^n}f_1(x_1)\cdots f_n(x_n)\,d\mu(x_1,\dots,x_n),
\end{equation}
for all $(f_1,\dots,f_n)\in \left(B_+(X)\right)^n$.
\end{thm}

\begin{proof}
Let $F:\left(B_+(X)\right)^n\to [0,\infty)$ be additive.  By Lemma~\ref{lemthm1} with $Y=B(X)$, $F$ extends to a positive multilinear functional on $\left(B(X)\right)^n$ that we shall also denote by $F$.  Since $F$ is positive and additive on $\left(C_+(X)\right)^n$, Theorem~\ref{thm1} implies that there is a finite Radon measure $\mu$ in $X^n$ such that (\ref{1aaa}) holds when $f_1,\dots,f_n\in C_+(X)$.  Define
\begin{equation}\label{1ah}
\widehat{F}(f_1,\dots,f_n)=\int_{X^n}f_1(x_1)\cdots f_n(x_n)\,d\mu(x_1,\dots,x_n),
\end{equation}
for all $(f_1,\dots,f_n)\in \left(B(X)\right)^n$.  We have to show that $F=\widehat{F}$ on $\left(B_+(X)\right)^n$.

Let $f_1,\dots,f_n\in B_+(X)$ and choose $M$ such that $f_i\le M$ for $i=1,\dots,n$. Let $\ee>0$.  Suppose that $i\in \{1,\dots,n\}$.  Let $\mu_i$ be the finite Radon measure in $X$ that is the projection of $\mu$ onto the $i$th copy of $X$ in the product $X^n$, i.e.,
$$\mu_i(E)=\mu(X\times\cdots\times X\times  E\times X\times\cdots\times X),$$
for all Borel sets $E$ in $X$.
By Lusin's theorem, there is a $g_i\in C_+(X)$ and a compact set $A_i$ in $X$ such that $g_i=f_i$ on $X\setminus A_i$ and $\mu_i(A_i)<\ee$.  We may also assume that $g_i\le M$.  Then $g_i-f_i=0$ on $X\setminus A_i$ and $|g_i-f_i|\le M$ on $A_i$, so
\begin{equation}\label{help4}
\int_{X}|g_i(x)-f_i(x)|\,d\mu_i(x)\le M\mu_i(A_i)<M\ee.
\end{equation}
Since $A_i$ is compact, we can choose $h_i\in C_+(X)$ such that $|g_i-f_i|\le h_i$ and
\begin{equation}\label{help1}
\int_{X}h_i(x)\,d\mu_i(x)<(M+1)\ee.
\end{equation}
By (\ref{Graf}), Lemma~\ref{addabs} with $Y=B(X)$, the fact that $F$ is increasing on  $\left(B_+(X)\right)^n$ (given by Lemma~\ref{addmon} with $Y=B(X)$), and (\ref{help1}), we obtain
\begin{eqnarray}\label{help2}
|F(g_1,\dots,g_n)-F(f_1,\dots,f_n)|
&\le &\sum_{i=1}^n |F(f_1,\dots,f_{i-1},g_i-f_i,g_{i+1},
\dots,g_n)|\nonumber\\
&\le &\sum_{i=1}^n F(f_1,\dots,f_{i-1},|g_i-f_i|,g_{i+1},
\dots,g_n)\nonumber\\
&\le & \sum_{i=1}^nF(M,\dots,M,h_i,M,\dots,M)\nonumber\\
&=&\sum_{i=1}^n\int_{X^n}M^{n-1}h_i(x_i)\,d\mu(x_1,\dots,x_n)\nonumber\\
&=&\sum_{i=1}^n\int_{X}M^{n-1}h_i(x_i)\,d\mu_i(x_i)<nM^{n-1}(M+1)\ee.
\end{eqnarray}
Noting that $\widehat{F}$ is additive on $\left(B(X)\right)^n$ by its definition, we can use (\ref{Graf}) again, (\ref{1ah}), and (\ref{help4}) to get
\begin{eqnarray}\label{help3}
|\widehat{F}(g_1,\dots,g_n)-\widehat{F}(f_1,\dots,f_n)|&= &\left|\sum_{i=1}^n \widehat{F}(f_1,\dots,f_{i-1},g_i-f_i,g_{i+1},
\dots,g_n)\right|\nonumber\\
&\le &\sum_{i=1}^n\int_{X}M^{n-1}|g_i(x_i)-f_i(x_i)|\,d\mu_i(x_i)< nM^{n}\ee.
\end{eqnarray}
Since $F(g_1,\dots,g_n)=\widehat{F}(g_1,\dots,g_n)$, (\ref{help2}) and (\ref{help3}) yield
$$|F(f_1,\dots,f_n)-\widehat{F}(f_1,\dots,f_n)|<nM^{n-1}(2M+1)\ee.$$
It follows that $F(f_1,\dots,f_n)=\widehat{F}(f_1,\dots,f_n)$ and hence that $F=\widehat{F}$ on $\left(B_+(X)\right)^n$.
\end{proof}

The following result is obtained from Theorem~\ref{thm3} in exactly the same fashion as Corollaries~\ref{corthm1} and~\ref{cor2thm1} were obtained from Theorem~\ref{thm1}.

\begin{cor}\label{corthm3}
If $F:\left(B_+(S^{n-1})\right)^n\to [0,\infty)$ is additive and vanishes when the supports of two of its arguments are disjoint, then there is a finite Radon measure $\mu$ in $S^{n-1}$ such that
$$
F(f_1,\dots,f_n)=\int_{S^{n-1}}f_1(u)\cdots f_n(u)\,d\mu(u),
$$
for all $(f_1,\dots,f_n)\in \left(B_+(S^{n-1})\right)^n$.  If in addition $F$ is rotation invariant, then there is a $c\ge 0$ such that
$$
F(f_1,\dots,f_n)=c\int_{S^{n-1}}f_1(u)\cdots f_n(u)\,du,
$$
for all $(f_1,\dots,f_n)\in \left(B_+(S^{n-1})\right)^n$.
\end{cor}

\section{Positive additive functionals on $\left({\mathcal{S}}_o^n\right)^n$
or $\left({\mathcal{S}}^n\right)^n$}\label{SorS}

In this section we draw conclusions from the results of the previous section by identifying a star body (or star set) $L$ in $\R^n$ with its radial function $\rho_L\in C_+(S^{n-1})$ (or $\rho_L\in B_+(S^{n-1})$, respectively).  Various properties of functionals on $\left({\mathcal{S}}_o^n\right)^n$
or $\left({\mathcal{S}}^n\right)^n$ can now be defined via those of the corresponding properties of functions on $\left(C_+(S^{n-1})\right)^n$ or $\left(B_+(S^{n-1})\right)^n$, respectively.  Thus we say that a functional $F$ on $\left({\mathcal{S}}_o^n\right)^n$ is {\em additive} if
\begin{eqnarray}\label{staraddd}
\lefteqn{F(L_1,\dots,L_{i-1},L_i\widetilde{+}M_i,L_{i+1},...,L_n)}\nonumber
&\\
&=&F(L_1,\dots,L_{i-1},L_i,L_{i+1},\dots,L_n)+F(L_1,\dots,L_{i-1},M_i,L_{i+1},\dots,L_n),
\end{eqnarray}
whenever $L_i,M_i\in {\mathcal{S}}_o^n$, $i=1,\dots,n$, {\em positive} if $F\ge 0$, and {\em rotation invariant} if
\begin{equation}\label{starrotd}
F(\phi L_1,\dots,\phi L_n)=F(L_1,...,L_n),
\end{equation}
for all $L_1,\dots,L_n\in {\mathcal{S}}_o^n$ and rotations $\phi$ of $S^{n-1}$.  The corresponding properties of a functional $F$ on $\left({\mathcal{S}}^n\right)^n$ are defined analogously.

Theorems~\ref{thm1} and~\ref{thm3} and Corollaries~\ref{corthm1}, \ref{cor2thm1}, and \ref{corthm3}, immediately yield the following result.

\begin{thm}\label{thm2}
Let $X={\mathcal{S}}_o^n$ or ${\mathcal{S}}^n$.

{\noindent{\rm{(i)}}} If $F:X^n\to [0,\infty)$ is additive, then there is a finite Radon measure $\mu$ in $\left(S^{n-1}\right)^n$ such that
$$
F(L_1,\dots,L_n)=\int_{\left(S^{n-1}\right)^n}\rho_{L_1}(u_1)\cdots \rho_{L_n}(u_n)\,d\mu(u_1,\dots,u_n),
$$
for all $L_1,\dots, L_n\in X$.

{\noindent{\rm{(ii)}}} If $F$ also vanishes when the intersection of two of the arguments is $\{o\}$, then there is a finite Radon measure $\mu$ in $S^{n-1}$ such that
$$
F(L_1,\dots,L_n)=\int_{S^{n-1}}\rho_{L_1}(u)\cdots \rho_{L_n}(u)\,d\mu(u),
$$
for all $L_1,\dots, L_n\in X$.

{\noindent{\rm{(iii)}}} If in addition to the previously assumed properties $F$ is also rotation invariant, then there is a $c\ge 0$ such that
$$
F(L_1,\dots,L_n)=c\widetilde{V}(L_1,\dots,L_n),
$$
for all $L_1,\dots, L_n\in X$.
\end{thm}

The following examples show that none of the assumptions in Theorem~\ref{thm2} can be omitted.

\begin{ex}\label{example1}
{\em For $L_i\in {\mathcal{S}}_o^n$, $i=1,\dots,n$ (or for $L_i\in {\mathcal{S}}^n$, $i=1,\dots,n$), define
$$F(L_1,\dots,L_n)={\mathcal{H}}^n\left(\cap_{i=1}^nL_i\right).$$
Then $F$ is rotation invariant and vanishes when the intersection of two of its arguments is $\{o\}$, but it is not additive.}
\end{ex}

\begin{ex}\label{example2}
{\em For $L_i\in {\mathcal{S}}_o^n$, $i=1,\dots,n$ (or for $L_i\in {\mathcal{S}}^n$, $i=1,\dots,n$), define
$$F(L_1,\dots,L_n)=\prod_{i=1}^n\int_{S^{n-1}}\rho_{L_i}(u)\,du.$$
Clearly, $F$ is additive and rotation invariant. However, it does not always vanish when the intersection of two of its arguments is $\{o\}$. For example, if the $L_i$'s are cones whose bases are disjoint spherical caps of positive radius, then $F(L_1,\dots,L_n)>0$.}
\end{ex}

\begin{ex}\label{example3}
{\em Let $M$ be any star body that is not a ball with center at the origin. For $L_i\in {\mathcal{S}}_o^n$, $i=1,\dots,n$ (or for $L_i\in {\mathcal{S}}^n$, $i=1,\dots,n$), define
$$F(L_1,\dots,L_n)=\int_{S^{n-1}}\rho_{L_1}(u)\cdots \rho_{L_n}(u)\rho_M(u)\,du.$$
Then $F$ is additive and vanishes when the intersection of two of its arguments is $\{o\}$, but it is not rotation invariant.}
\end{ex}

\section{Real-valued additive functionals.}

The positivity of $F$ was used in an essential way in Lemma~\ref{addmon}, in which the fact that $F:Y_+^n\to [0,\infty)$ is increasing was deduced from its additivity.  However, all the main results in Sections~\ref{funsonCB} and~\ref{SorS} hold for real-valued functionals if it is assumed in addition that they are increasing.  Indeed, the simple observation that if $F: Y_+^n \to \R$ is additive and increasing, then $F \ge 0$, allows all the proofs go through as before.  In particular, we have the following result.

\begin{thm}\label{thm2real}
Let $X={\mathcal{S}}_o^n$ or ${\mathcal{S}}^n$. If $F:X^n\to \R$ is additive, increasing, rotation invariant, and vanishes when the intersection of two of the arguments is $\{o\}$, then there is a $c\ge 0$ such that
$$
F(L_1,\dots,L_n)=c\widetilde{V}(L_1,\dots,L_n),
$$
for all $L_1,\dots, L_n\in X$.
\end{thm}

Note that here $F$ is increasing if it is increasing in each argument with respect to set inclusion.  This is compatible with our previous use of the term, since if $X={\mathcal{S}}_o^n$ or ${\mathcal{S}}^n$ and $L,M\in X$, then $L\subset M$ if and only if $\rho_L\le \rho_M$.

None of the assumptions in Theorem~\ref{thm2real} can be omitted.  Indeed, all the functionals in Examples~\ref{example1},~\ref{example2}, and~\ref{example3} are increasing, showing that none of the other assumptions can be dropped.  If we define $F(L_1,\dots,L_n)=-\widetilde{V}(L_1,\dots,L_n)$, then of course $F$ is not increasing but retains the other properties assumed in Theorem~\ref{thm2real}.

However, an intriguing possibility arises, namely, that without assuming that $F$ is increasing in Theorem~\ref{thm2real}, the result holds with the weaker conclusion that there is a $c\in \R$ such that
$F(L_1,\dots,L_n)=c\widetilde{V}(L_1,\dots,L_n)$,
for all $L_1,\dots, L_n\in X$.  The following example addresses this question.

\begin{ex}\label{examplereal}
{\em Let $X={\mathcal{S}}_o^n$ or ${\mathcal{S}}^n$.  Assuming the Axiom of Choice, there is an additive function $h:\R\to \R$ which is not linear (see, for example, \cite[Section~7.3]{Cie97}).  Define $F:X^n\to \R$ by
$$F(L_1,\dots,L_n)=h\left(\widetilde{V}(L_1,\dots,L_n)\right),$$
for $L_i\in X$, $i=1,\dots,n$.  It is easy to check that $F$ is additive, rotation invariant, and vanishes when the intersection of two of its arguments is $\{o\}$.  (The latter property requires $h(0)=0$, a consequence of the additivity of $h$.)  However, there is no $c\in \R$ such that $F(L_1,\dots,L_n)=c\widetilde{V}(L_1,\dots,L_n)$.  If there were, then given $t\ge 0$, we could choose $L_1(t),\dots,L_n(t)\in X$ such that $\widetilde{V}(L_1(t),\dots,L_n(t))=t$ (for example by taking $L_i(t)=\left(t/{{\mathcal{H}}^n}(B^n)\right)^{1/n}B^n$, for $i=1,\dots,n$, where $B^n$ is the unit ball in $\R^n$), leading to
$h(t)=F(L_1(t),\dots,L_n(t))=ct$, for all $t\ge 0$.  Then for $t<0$, we have $h(t)=h(0)-h(-t)=ct$ by the additivity of $h$, so $h$ is linear on $\R$, a contradiction.}
\end{ex}

We remark that the previous example may be adapted to form a small observation regarding the paper \cite{MilS11} by Milman and Schneider on characterizing the mixed volume.  Namely, with $h$ as in the previous example, the functional $F:\left({\mathcal{K}}^n\right)^n\to \R$ defined on $n$-tuples of compact convex sets in $\R^n$ by
$F(K_1,\dots,K_n)=h(V(K_1,\dots,K_n))$, for $K_i\in {\mathcal{K}}^n$, $i=1,\dots,n$, is additive and vanishes if two of its arguments are parallel line segments, but $F$ is not a real constant multiple of the mixed volume.  This shows that this weaker conclusion to \cite[Theorem~2]{MilS11} cannot be obtained in ZFC if the assumption that $F$ is increasing is omitted.

Additive nonlinear functions from $\R$ to $\R$ can be constructed via a Hamel basis, which in turn is constructed using the Axiom of Choice.  It is known, however, that it is consistent with Zermelo-Fraenkel set theory ZF that all additive functions from $\R$ to $\R$ are linear. This follows from Solovay's model \cite{She84} of ZF in which every set of reals is Baire measurable, together with the fact that any additive, Baire-measurable function from $\R$ to $\R$ must be linear.  (The latter fact is proved in the same way as the well-known result that any additive, Lebesgue-measurable function from $\R$ to $\R$ must be linear.)  We leave open the question as to whether it is consistent with ZF that Theorem~\ref{thm2real} holds for some $c\in\R$ without the assumption that $F$ is increasing, as well as the corresponding question regarding \cite[Theorem~2]{MilS11}.

\section*{Appendix: A direct approach to a case of Theorem~\ref{thm2}}\label{Direct}

It is perhaps worth remarking that Theorem~\ref{thm2}(iii) can be proved directly, that is, without using Proposition~\ref{prp0}, at least in the case when $X={\mathcal{S}}^n$ and the slightly stronger assumption is made that $F$ vanishes when the intersection of two its arguments has ${\mathcal{H}}^n$-measure zero.  Here we outline how this may be done.  A little terminology is needed.

As in  \cite{Kla96}, we define the \emph{star hull} of a set $A$ in $\R^n$ by
$$\st A=\{tx:x\in A, 0\leq t\leq1\}.$$
If $\alpha>0$ and $A$ is a Borel set in $S^{n-1}$, the set $C=\alpha\,\st A$ will be called a  \textit{cone} of base $A$ and radius $\alpha$. Note that $C$ is a star set and $\rho_{C}=\alpha 1_A$.

A \textit{polycone} is a finite union of cones.  If $P$ is a nontrivial polycone, there are unique $\alpha_j>0$ and disjoint Borel sets $A_j\subset S^{n-1}$, $j=1,\dots,m$, such that
$$\rho_P=\sum_{j=1}^m\alpha_j1_{A_j}.$$
(Compare \cite[Proposition~2.12]{Kla96}.  No proof is given, but the argument is straightforward.)  If $C_j=\alpha_j\,\st A_j$, then $P=\cup_{j=1}^mC_j$ expresses the polycone $P$ as the union of cones $C_j$, $j=1,\dots,m$, that meet only at the origin.

With this in hand, we can sketch the proof. If $F:(\mathcal{S}^n)^n\to[0,\infty)$ is additive, then via Lemmas~\ref{addmon} and~\ref{addhomon}, $F$ may be assumed to be also increasing and positively homogeneous whenever these properties are required. Suppose that $F$ vanishes whenever the intersection of two of its arguments is $\{o\}$ and define $\mu(A)=F(\st A,\dots,\st A)$, for each Borel set $A$ in $S^{n-1}$. Then one can show that $\mu$ is a valuation, i.e., that
$$
\mu(A\cup B)+\mu(A\cap B)=\mu(A)+\mu(B),
$$
for all Borel sets $A,B\subset S^{n-1}$.  Now if $F$ vanishes when the intersection of two its arguments has ${\mathcal{H}}^n$-measure zero and is rotation invariant, then $\mu$ is a rotation invariant valuation on the Borel sets in $S^{n-1}$ that vanishes on sets of ${\mathcal{H}}^{n-1}$-measure zero.  The restriction of $\mu$ to the spherical convex polytopes in $S^{n-1}$ is therefore a nonnegative, rotation invariant valuation which is also simple, meaning that it vanishes on spherical convex polytopes in $S^{n-1}$ that are not full dimensional.  A result of Schneider \cite[Theorem~6.2]{Sch78} implies that there is a $\lambda\ge 0$ such $\mu(A)=\lambda{\mathcal{H}}^{n-1}(A)$ whenever $A$ is a spherical convex polytope in $S^{n-1}$. As is shown in \cite[p.~226]{McS}, this also holds whenever $A$ is a Borel set in $S^{n-1}$.  From this and the positive homogeneity of $F$, it is easy to conclude that there is a $c\ge 0$ such that
\begin{equation}\label{eq:characterization_volumes}
F(L,\dots,L)=c{\mathcal{H}}^n(L),
\end{equation}
for any polycone $L$.

The next step is to show that if $A_i$ is a Borel set in $S^{n-1}$, $\alpha_i> 0$,  and $C_i$ is the cone with base $A_i\subset S^{n-1}$ and radius $\alpha_i$, $i=1,\dots,n$, then
\begin{equation}\label{eq:Fcones}
F(C_1,\dots,C_n)=\frac{\alpha_1\cdots\alpha_n}{(\min\{\alpha_1,\dots, \alpha_n\})^n}F\left(L,\dots,L\right),
\end{equation}
where $L=\cap_{i=1}^nC_i$.  This is done by a standard disjointification argument, again using the positive homogeneity of $F$.  The rotation invariance of $F$ is not needed for the latter step, but may now be invoked, together with (\ref{eq:characterization_volumes}) and (\ref{eq:Fcones}), to yield that there is a $c\ge 0$ such that
\begin{equation}\label{eq:characterization_cones}
F(C_1,\dots,C_n)=c\widetilde{V}(C_1,\dots,C_n),
\end{equation}
for cones $C_1,\dots,C_n$.

Now, using (\ref{eq:characterization_cones}) and the additivity of $F$, it is routine to show that (\ref{eq:characterization_cones}) holds when the $C_i$'s are polycones.  The final step is to show that (\ref{eq:characterization_cones}) holds when $C_i$ is replaced by a general star set $L_i$, $i=1,\dots,n$.  This is achieved by the usual uniform approximation of the nonnegative, bounded Borel function $\rho_{L_i}$ by simple nonnegative Borel functions (see, for example, \cite[Theorem~1.17]{Rud87}) and using the fact that $F$ is increasing and the monotone convergence theorem.

\end{document}